\documentclass[a4paper,reqno]{amsart}
\usepackage{a4wide,mathptmx} 
\usepackage{amssymb}

\providecommand{\cal}[1]{\mathcal{#1}}

\newcommand{\gen}{\mathbf{A}}

\newcommand{\B}{{\mathbb{B}}}
\newcommand{\C}{{\mathbb{C}}}
\newcommand{\N}{{\mathbb{N}}}
\newcommand{\R}{{\mathbb{R}}}
\newcommand{\Rn}{{\mathbb{R}}^{n}}

\newcommand{\dv}{\operatorname{div}}
\newcommand{\grad}{\operatorname{grad}}
\newcommand{\im}{\operatorname{i}}
\newcommand{\lap}{\operatorname{\Delta}}
\newcommand{\mlap}{-\!\operatorname{\Delta}}
\newcommand{\scal}[2]{(\,#1\,|\, #2\,)}
\newcommand{\set}[2]{\{\,#1\bigm| #2\,\}}
\newcommand{\Set}[2]{\bigl\{\,#1\bigm| #2\,\bigr\}}
\newcommand{\vvvert}{{|\hspace{-1.6pt}|\hspace{-1.6pt}|}}
\newcommand{\dual}[2]{\ensuremath{\langle{#1},{#2}\rangle}} 
\newcommand{\Dual}[2]{\ensuremath{\big\langle{#1},{#2}\big\rangle}}

\newtheorem{theorem}{Theorem}
\newtheorem{proposition}{Proposition}
\newtheorem{lemma}{Lemma}

\theoremstyle{definition}

\theoremstyle{remark}
\newtheorem{remark}{Remark}

\begin{document}
\title[Well-Posed Final Value Problems for Coercive Operators]{Well-Posed Final Value Problems and Duhamel's Formula for Coercive Lax--Milgram Operators}

\author{Jon Johnsen}
\address{Department of Mathematics, Aalborg University, Skjernvej 7A, DK-9220 Denmark}
\email{jjohnsen@math.aau.dk}
\thanks{~
\\[9\jot]{\tt Appeared in Electronic Research Archives, vol.\ 27 (2019), 20--36; doi:10.3934/era2019008}}

\begin{abstract}This paper treats parabolic final
  value problems generated by coercive Lax--Milgram operators, and well-posedness is proved for this large class.
  The result is obtained by means of an isomorphism between Hilbert spaces containing the data and solutions.  
  Like for elliptic generators, the data space is the graph normed domain of an unbounded operator that maps final states
  to the corresponding initial states, and the resulting compatibility condition extends to the coercive context.
  Lax--Milgram operators in vector distribution spaces is the main framework, but the crucial
  tool that analytic semigroups always are invertible in the class of closed operators is extended
  to unbounded semigroups, and this is shown to yield a Duhamel formula for the Cauchy problems in
  the set-up.  
  The final value heat conduction problem with the homogeneous Neumann boundary condition on a smooth open set is
  also proved to be well posed in the sense of Hadamard.
\end{abstract}


\maketitle

\section{Introduction}
\thispagestyle{empty}

Well-posedness of final value problems
for a large class of parabolic differential equations was recently obtained in a joint work of the author
and given an ample description for a broad audience in \cite{ChJo18ax},
after the announcement in \cite{ChJo18}. The present paper substantiates the indications made in
the concise review \cite{JJ19}, namely, that  the abstract parts in 
\cite{ChJo18ax} extend from $V$-elliptic Lax--Milgram operators $A$ to those that are merely
$V$-\emph{coercive}---despite  that such $A$ may be non-injective.

As an application, the final value heat conduction problem with the homogeneous Neumann condition is
shown to be well-posed.

\bigskip 

The basic analysis is made for a
Lax--Milgram operator $A$ defined in $H$ from a $V$-coercive sesquilinear form $a(\cdot,\cdot)$ 
in a Gelfand triple, i.e.,
three separable, densely injected Hilbert spaces $V\hookrightarrow H\hookrightarrow V^*$ 
having the norms $\|\cdot\|$, $|\cdot|$ and $\|\cdot\|_*$, respectively. Hereby $V$ is the form
domain of $a$; and $V^*$ the antidual  of $V$.  
Specifically there are constants $C_j>0$ and $k\in\R$ 
such that all $u, v\in V$ satisfy $\| v\|_*\le C_1|v|\le C_2 \| v\|$ and 
\begin{equation} \label{coerciv-id}
  |a(u,v)|\le C_3\|u\|\,\|v\|\,, \qquad \Re a(v,v)\ge C_4\|v\|^2-k|u|^2. 
\end{equation}
In fact, $D(A)$ consists of those $u\in V$ for
which $a(u,v)=\scal{f}{v}$ for some $f\in H$ holds for all $v\in V$, and $Au=f$;
hereby $\scal{u}{v}$ denotes the inner product in $H$.
There is also an extension $A\in\B(V,V^*)$ given by $\dual{Au}{v}=a(u,v)$ for
$u,v\in V$. This is uniquely determined as $D(A)$ is dense in $V$.

Both $a$ and $A$ are referred to as $V$-elliptic if the above holds for $k=0$; then 
$A\in\B(V,V^*)$ is a bijection.
One may consult the book of Grubb \cite{G09} or that of Helffer \cite{Hel13}, or \cite{ChJo18ax}, for more
details on the set-up and basic properties of the unbounded, but closed operator $A$ in $H$.
Especially $A$ is self-adjoint in $H$  if and only if $a(v,w)=\overline{a(w,v)}$, which is not assumed;
$A$ may also be nonnormal in general.

In the framework of such a triple $(H,V,a)$, the general final value problem is the following: \emph{for given data
$f\in L_2(0,T; V^*)$ and $u_T\in H$,
determine the  $u\in{\cal D}'(0,T;V)$ such that}
\begin{equation}
  \label{fvA-intro}
  \left.
  \begin{aligned}
  \partial_tu +Au &=f  &&\quad \text{in ${\cal D}'(0,T;V^*)$}
\\
  u(T)&=u_T &&\quad\text{in $H$}
\end{aligned}
\right\}
\end{equation}
By definition of Schwartz' vector distribution space ${\cal D}'(0,T;V^*)$ as the space of continuous linear maps 
$C_0^\infty(]0,T[)\to V^*$, cf.\ \cite{Swz66}, the above equation means that 
for every scalar test function $\varphi\in C_0^\infty(]0,T[)$
the identity $\dual{u}{-\varphi'}+\dual{A u}{\varphi}=\dual{f}{\varphi}$ holds in $V^*$.

As is well known, a wealth of parabolic Cauchy problems with homogeneous boundary 
conditions have been treated via triples $(H,V,a)$ and the ${\cal D}'(0,T;V^*)$
set-up in \eqref{fvA-intro}; cf.\ the
work of Lions and Magenes~\cite{LiMa72}, Tanabe~\cite{Tan79}, Temam~\cite{Tem84}, Amann
\cite{Ama95} etc. 

The theoretical analysis made in \cite{ChJo18,ChJo18ax, JJ19} shows that, in the $V$-elliptic case,
the problem in \eqref{fvA-intro} is well posed, i.e., it has
\emph{existence, uniqueness} and \emph{stability} of a solution $u\in X$ for given data
$(f,u_T)\in Y$, in certain Hilbertable spaces $X$, $Y$ that were described explicitly. 
Hereby the data space $Y$ is defined in terms of a particular compatibility condition, which was
introduced for the purpose in \cite{ChJo18,ChJo18ax}. 
More precisely, there is even a linear homeomorphism $X\longleftrightarrow Y$, which yields
well-posedness in a strong form.

This has seemingly closed a gap in the theory, which had remained since the 1950's,
even though the well-posedness is decisive for the interpretation and accuracy of 
numerical schemes for the problem (the work of John~\cite{John55} was pioneering,
but also Eld\'en \cite{Eld87} could be mentioned).
In rough terms, the results are derived from a useful structure on the reachable set for a
general class of parabolic differential equations.

The main example treated in \cite{ChJo18,ChJo18ax} is the heat conduction problem of
characterising the 
$u(t,x)$ that in a $C^\infty$-smooth bounded open set 
$\Omega\subset\Rn$ with boundary $\Gamma=\partial\Omega$
fulfil the equations (for $\Delta=\partial_{x_1}^2+\dots+\partial_{x_n}^2$),
\begin{equation}  \label{heat-intro}
\left.
\begin{aligned}
  \partial_tu(t,x)-\lap u(t,x)&=f(t,x) &&\quad\text{for $t\in\,]0,T[\,$,  $x\in\Omega$}
\\
   u(t,x)&=g(t,x) &&\quad\text{for $t\in\,]0,T[\,$, $x\in\Gamma$}
\\
  u(T,x)&=u_T(x) &&\quad\text{for $x\in\Omega$}
\end{aligned}
\right\}
\end{equation}
An area of interest of this could be a nuclear power plant 
hit by a power failure at $t=0$: after power is regained at $t=T>0$, and 
the reactor temperatures $u_T(x)$ are measured, a
calculation backwards in time could possibly settle whether at some $t_0<T$ the temperatures $u(t_0,x)$
could cause damage to the fuel rods. 

However, the Dirichlet condition $u=g$ at the boundary
$\Gamma$ is of limited physical importance, so an extension to, e.g., the Neumann
condition, which represents controlled heat flux at $\Gamma$,
makes it natural to work out an extension to $V$-coercive Lax--Milgram operators $A$. 

In this connection it should be noted that when $A$ is $V$-coercive (that is, satisfies \eqref{coerciv-id}
only for some $k>0$), it is possible that $0\in\sigma(A)$, the spectrum of $A$, for example because
$\lambda=0$ is an eigenvalue of $A$. In fact, this is the case for the Neumann realisation
$\mlap_N$, which has the space of constant functions 
$\C1_\Omega$ as the null space.
Well-posedness is obtained for the heat problem \eqref{heat-intro} with a replacement of the Dirichlet condition
by the homogeneous Neumann condition in Section~\ref{Neu-sect} below.

\bigskip
At first glance, it may seem surprising that the possible non-injectivity of the coercive operator
$A$ is \emph{inconsequential} for the well-posedness of the final value problem \eqref{fvA-intro}.
In particular this means that the backward uniqueness---$u(T)=0$ in
$H$ implies $u(t)=0$ in $H$ for $0\le t<T$---of the equation $u'+Au=f$ will hold
regardless of whether $A$ is injective or not. This can be seen from the extensions of the abstract
theory given below; in particular when the results are applied in Section~\ref{Neu-sect} to the case $A=\mlap_N$.

The point of departure is to make a comparison of \eqref{fvA-intro} with the corresponding Cauchy
problem for the equation $u'+Au=f$.
For this it is classical to seek solutions $u$ in the Banach space 
\begin{equation}
  \begin{split}
  X=&L_2(0,T;V)\bigcap C([0,T];H) \bigcap H^1(0,T;V^*),
\\
  \|u\|_X=&\big(\int_0^T \|u(t)\|^2\,dt+\sup_{0\le t\le T}|u(t)|^2+\int_0^T (\|u(t)\|_*^2   +\|u'(t)\|_*^2)\,dt\big)^{1/2}.   
  \end{split}
  \label{eq:X}
\end{equation}
In fact, the following result is essentially known from the work of Lions and Magenes \cite{LiMa72}:
 
\begin{proposition}  \label{LiMa-prop}
Let $V$ be a separable Hilbert space with $V \subseteq H$ algebraically, topologically and densely,
and let $A$ denote the Lax--Milgram operator induced by a $V$-coercive, bounded 
sesquilinear form on $V$, as well as its extension $A\in\B(V,V^*)$.
When $u_0 \in H$ and $f \in L_2(0,T; V^*)$ are given, then there is 
a uniquely determined solution $u$ belonging to $X$, cf.\ \eqref{eq:X},
of the Cauchy problem 
\begin{equation}
  \left.
  \begin{aligned}
  \partial_tu +Au &=f  \quad \text{in ${\cal D}'(0,T;V^*)$}
\\
  \qquad u(0)&=u_0 \quad\text{in $H$}
  \end{aligned}
  \right\}
  \label{ivA-intro}
\end{equation}
The solution operator $(f,u_0)\mapsto u$ is continuous $L_2(0,T; V^*)\oplus H\to X$, 
and problem \eqref{ivA-intro} is well-posed.
\end{proposition}

Remarks on the classical reduction from the $V$-coercive case to the elliptic one will follow in
Section~\ref{aproof-sect}. 
The stated continuity of the solution operator is well known to the experts.
But for the reader's convenience,  in Proposition~\ref{pest-prop} below, the continuity is shown by explicit
estimates using Gr{\"o}nwall's lemma; these may be of independent interest.

Whilst the below expression for the solution hardly is surprising at all, 
it has seemingly not been obtained hitherto in the present context of $V$-coercive Lax--Milgram
operators $A$ and general triples $(H,V,a)$:
 
\begin{proposition}   \label{Duhamel-prop}
The unique solution $u$ in $X$ provided by Proposition~\ref{LiMa-prop} is given by Duhamel's
formula,
\begin{equation} \label{u-id}
  u(t) = e^{-tA}u_0 + \int_0^t e^{-(t-s)A}f(s) \,ds \qquad\text{for } 0\leq t\leq T.
\end{equation}
Here each of the three terms belongs to $X$.
\end{proposition}
As shown in Section~\ref{aproof-sect} below, it suffices for \eqref{u-id} to reinforce the
classical integration factor technique by \emph{injectivity} of the semigroup $e^{-tA}$.

In fact, it is exploited in \eqref{u-id} and throughout that $-A$ generates an analytic semigroup $e^{-zA}$ in
$\B(H)$. 
As a consequence of the analyticity, the family of operators $e^{-zA}$
was shown in \cite{ChJo18ax} to consist of \emph{injections} on $H$ in case $A$ is $V$-elliptic.
This extends to general $V$-coercive $A$, as accounted for in Proposition~\ref{inj-prop} below.
Hence $e^{-tA}$ also in the present paper has the inverse $e^{tA}:=(e^{-tA})^{-1}$ for $t>0$.

For $t=T$, the Duhamel formula \eqref{u-id} now obviously yields a 
\emph{bijection} $u(0)\longleftrightarrow u(T)$ between the initial
and terminal states (for fixed $f$), as one can solve for $u_0$ by means of the inverse $e^{TA}$. 
In particular backwards uniqueness of the solutions to $u'+Au=f$ holds in the large class $X$.

\bigskip
Returning to the final value problem \eqref{fvA-intro} it would be natural to seek
solutions $u$ in the same space $X$. 
This turns out to be possible only when the data $(f, u_T)$ are subjected to substantial further conditions. 

To formulate these, it is noted that the above inverse $e^{tA}$ enters the theory through its
domain, which  in the algebraic sense simply is a range, namely $D(e^{tA})=R(e^{-tA})$; but this
domain has the structural advantage of being a Hilbert space under 
the graph norm $\|u\|=(|u|^2+|e^{tA}u|^2)^{1/2}$.

For $t=T$ the domains $D(e^{T A})$ have a decisive role in the well-posedness result below, where
condition \eqref{eq:cc-intro} is a fundamental clarification for the final value problem
in \eqref{fvA-intro} and the parabolic problems it represents.

Another ingredient in \eqref{eq:cc-intro} is the full yield $y_f$ of the source term
$f\colon \,]0,T[\, \to V^*$, namely
\begin{equation} \label{yf-eq}
  y_f= \int_0^T e^{-(T-t)A}f(t)\,dt.
\end{equation}
Hereby it is used that $e^{-tA}$ extends to an analytic semigroup in  $V^*$,
as the extension $A\in\B(V,V^*)$ is an unbounded operator in the Hilbertable space $V^*$ satisfying the necessary
estimates (cf.\ Remark~\ref{domain-rem}; and also \cite[Lem.\ 5]{ChJo18ax} for the extension).
So $y_f$ is a priori a vector in $V^*$, but in fact $y_f$ lies in $H$ 
as Proposition~\ref{Duhamel-prop} shows it equals 
the final state of a solution in $C([0,T],H)$ of a Cauchy problem having $u_0=0$. 

These remarks on $y_f$ make it clear that in the following main result of the paper---which
relaxes the assumption of $V$-ellipticity in \cite{ChJo18,ChJo18ax} to $V$-coercivity---the
difference in \eqref{eq:cc-intro} is a member of  $H$:

\begin{theorem} \label{intro-thm}
  Let $A$ be a $V$-\emph{coercive} Lax--Milgram operator defined from a triple $(H,V,a)$ as above.
  Then the abstract final value problem \eqref{fvA-intro} has a solution $u(t)$ belonging the space
  $X$ in \eqref{eq:X}, if and only if the data $(f,u_T)$ belong to the subspace 
  \begin{equation}
    Y\subset L_2(0,T; V^*)\oplus H
  \end{equation}
  defined  by the condition  
  \begin{equation}
    \label{eq:cc-intro}
    u_T-\int_0^T e^{-(T-t)A}f(t)\,dt \ \in\  D(e^{TA}).
  \end{equation}  
In the affirmative case, the solution $u$ is uniquely determined in $X$ and 
\begin{equation}
  \label{eq:Y-intro}
  \begin{split}
      \|u\|_{X}& \le
  c \Big(|u_T|^2+\int_0^T\|f(t)\|_*^2\,dt+\Big|e^{TA}\big(u_T-\int_0^Te^{-(T-t)A}f(t)\,dt\big)\Big|^2\Big)^{1/2}
\\
   &=c \|(f,u_T)\|_Y,
  \end{split}
\end{equation}
whence the solution operator $(f,u_T)\mapsto u$ is continuous $Y\to X$. Moreover,
\begin{equation} \label{eq:fvp_solution}
  u(t) = e^{-tA}e^{TA}\Big(u_T-\int_0^T e^{-(T-t)A}f(t)\,dt\Big) + \int_0^t e^{-(t-s)A}f(s) \,ds,
\end{equation}
where all terms belong to $X$ as functions of $t\in[0,T]$, and the difference in
\eqref{eq:cc-intro} equals $e^{-TA}u(0)$ in $H$.
\end{theorem}

The norm on the data space $Y$ in \eqref{eq:Y-intro} is seen at once to be the graph norm of the composite map
\begin{equation}
  L_2(0,T; V^*)\oplus H \xrightarrow[\;]{\quad \Phi\quad} H\xrightarrow[\;]{\quad e^{TA}\quad} H
\end{equation}
given by $(f,u_T)\ \mapsto u_T-y_f\ \mapsto \ e^{TA}(u_T-y_f)$ and $\Phi(f,u_T)=u_T-y_f$.

In fact,  the solvability criterion \eqref{eq:cc-intro} means that $e^{TA}\Phi$ must be defined at
$(f,u_T)$, so the data space $Y$ is its domain. Being an inverse, $e^{TA}$ is a closed operator in $H$, and so is
$e^{TA}\Phi$; hence $Y=D(e^{TA}\Phi)$ is complete. Now, since in \eqref{eq:Y-intro} the Banach space $V^*$
is Hilbertable, so is $Y$.

Thus the unbounded operator $e^{TA}\Phi$ is a key ingredient in the rigorous treatment of the
final value problem \eqref{fvA-intro}.
In control theoretic terms, the role of $e^{TA}\Phi$ is to provide the unique initial state given by
\begin{equation}
  u(0)=e^{TA}\Phi(f,u_T)=e^{TA}(u_T-y_f),
\end{equation}
which is steered by $f$ to the final state $u(T)=u_T$ at time $T$;
cf.\ the Duhamel formula \eqref{u-id}.

Criterion \eqref{eq:cc-intro} is a generalised \emph{compatibility} condition on
the data $(f,u_T)$; such conditions have long been known in the theory of parabolic problems, cf.\
Remark~\ref{GS-rem}.
The presence of $e^{-(T-t)A}$ and the integral over $[0,T]$ makes
\eqref{eq:cc-intro} \emph{non-local} in both space and time. This
aspect is further complicated by the reference to $D(e^{TA})$, which for larger
final times $T$ typically gives increasingly stricter conditions:

\begin{proposition}
If the spectrum $\sigma(A)$ of $A$ is not contained in the strip
$\set{z\in\C}{-k\le \Re z\le k}$, whereby $k$ is the constant from \eqref{coerciv-id}, then 
the domains $D(e^{tA})$ form a strictly descending chain, that is,
\begin{equation} \label{dom-intro}
 H\supsetneq D(e^{tA})\supsetneq D(e^{t' A}) \qquad\text{ for  $0<t<t'$}.
\end{equation}
\end{proposition}

This results from the injectivity of $e^{-tA}$ via well-known facts for semigroups 
reviewed in \cite[Thm.\ 11]{ChJo18ax} (with reference to \cite{Paz83}). In fact, the arguments
given for $k=0$ in \cite[Prop.\ 11]{ChJo18ax} apply mutatis mutandis.

Now, \eqref{u-id} also shows that $u(T)$ has two 
radically different contributions, even if $A$ has nice properties.
First, for $t=T$ the integral equals $y_f$, which can be
\emph{anywhere} in $H$.
Indeed, $f\mapsto y_f$ is a continuous \emph{surjection} $y_f\colon L_2(0,T;V^*)\to H$.
This was shown for $k=0$ via the Closed Range Theorem in \cite[Prop.\ 5]{ChJo18ax},
and for $k>0$ surjectivity follows from this case as
$e^{-(T-s)A}f(s)=e^{-(T-s)(A+kI)}e^{-sk}f(s)$ in \eqref{yf-eq}, whereby $A+kI$ is $V$-elliptic and $f\mapsto
e^{-sk}f$ is a bijection on $L_2(0,T;V^*)$.

Secondly, $e^{-tA}u(0)$ solves $u'+Au=0$, and for $u(0)\ne0$ and $V$-elliptic  $A$ it is a
precise property in non-selfadjoint dynamics that the ``height'' $h(t)= |e^{-tA}u(0)|$ is
\begin{align*}
  &\text{strictly positive ($h>0$)},\\ &\text{strictly decreasing ($h'<0$)},\\ &\text{\emph{strictly convex}
    ($\Leftarrow h''>0$)}  .
\end{align*}
Whilst this holds if $A$ is self-adjoint or normal, it was emphasized in \cite{ChJo18ax}
that it suffices that $A$ is just hyponormal (i.e., $D(A)\subset D(A^*)$ and $|Ax|\ge|A^*x|$ for
$x\in D(A)$, following Janas \cite{Jan94}). Recently this was followed up by the author in \cite{18logconv}, where the stronger 
logarithmic convexity of $h(t)$ was proved \emph{equivalent} to the formally weaker property of $A$
that, for $x\in D(A^2)$,
\begin{equation} \label{Alogconv-id}
  2(\Re\scal{A x}{x})^2\le \Re\scal{A^2x}{x}|x|^2+|A x|^2|x|^2 .
\end{equation}
For $V$-coercive $A$ only the strict decrease may need to be relinquished. Indeed, the strict
positivity $h(t)>0$ follows by the injectivity of $e^{-tA}$ in Proposition~\ref{inj-prop} below.
Moreover, the characterisation in \cite[Lem.\ 2.2]{18logconv}
of the log-convex $C^2$-functions $f(t)$ on $[0,\infty[\,$ as the solutions of
the differential inequality $f''\cdot f\ge (f')^2$  and the resulting criterion for $A$ in
\eqref{Alogconv-id} apply \emph{verbatim} to the coercive case; hereby the differential calculus in
Banach spaces is exploited in a classical derivation of  the formulae for $u(t)=e^{-tA}u(0)$,
\begin{align}
  h'(t)&=-\frac{\Re\scal{Au(t)}{u(t)}}{|u(t)|},
\\
  h''(t)&=\frac{\Re\scal{A^2u(t)}{u(t)}+|Au(t)|^2}{|u(t)|}
            -\frac{(\Re\scal{Au(t)}{u(t)})^2}{|u(t)|^3}.
\end{align}
But it is due to the strict positivity $|e^{-tA}u(0)|>0$ for $t\ge0$ in the denominators that
the expressions make sense, so injectivity of $e^{-tA}$ also enters crucially at this point.
Similarly the singularity of $|\cdot|$ at the origin poses no problems for the
mere differentiation of $h(t)$. 
Therefore it is likely that the natural formulas for $h'$, $h''$ have not been rigorously
proved before \cite{JJ19}. These remarks also shed light on the usefulness of Proposition~\ref{inj-prop} below.

However, the stiffness intrinsic to \emph{strict} convexity, hence to log-convexity, corresponds
well with the fact that
$u(T)=e^{-TA}u(0)$ in any case is confined to a dense, but very small space, as by the analyticity
\begin{equation}
  \label{DAn-cnd}
  u(T)\in \textstyle{\bigcap_{n\in\N}} D(A^n).
\end{equation}
For $u'+Au=f$, the possible 
$u_T$ will hence be a sum of some arbitrary $y_f\in H$ and a
stiff term $e^{-TA}u(0)$. Thus $u_T$ can be prescribed in the affine space
$y_f+D(e^{TA})$. As any $y_f\ne0$ will shift $D(e^{TA})\subset H$
in an arbitrary direction, $u(T)$ can be expected \emph{anywhere} in $H$ (unless $y_f\in D(e^{TA})$ is known).
So neither \eqref{DAn-cnd} nor $u(T)\in D(e^{TA})$ can be expected to hold if $y_f\ne0$---not even
if $|y_f|$ is much smaller than $|e^{-TA}u(0)|$. Hence it seems best for final value problems to
consider inhomogeneous problems from the outset.

\begin{remark}
To give some backgound, two classical observations for the homogeneous case $f=0$, $g=0$ in
\eqref{heat-intro} are  recalled. First there is the smoothing effect for $t>0$ of parabolic Cauchy
problems, which means that $u(t,x)\in C^\infty(\,]0,T]\times\overline\Omega)$ whenever $u_0\in L_2(\Omega)$. 
(Rauch \cite[Thm.\ 4.3.1]{Rau91} has a version for $\Omega=\Rn$; Evans
\cite[Thm.\ 7.1.7]{Eva10} gives the stronger result $u\in C^\infty([0,T]\times\overline\Omega)$ when
$f\in C^\infty([0,T]\times\overline\Omega)$, $g=0$ and $u_0\in C^\infty(\overline\Omega)$ fulfill the
classical compatibility conditions at $\{0\}\times\Gamma$---which for $f=0$, $g=0$ gives the $C^\infty$
property on $[\varepsilon, T]\times\overline{\Omega}$ for any $\varepsilon>0$, hence on  $\,]0, T]\times\overline{\Omega}$).
Therefore $u(T,\cdot)\in C^\infty(\overline\Omega)$; whence \eqref{heat-intro}
with $f=0$, $g=0$ cannot be solved if $u_T$ is prescribed arbitrarily in $L_2(\Omega)$. But this just
indicates an asymmetry in the properties of the initial and final value problems.

Secondly, there is a phenomenon of $L_2$-instability in case $f=0$, $g=0$ in \eqref{heat-intro}, which perhaps was first described by
Miranker  \cite{Miranker61}. The instability is found via the Dirichlet realization
of the Laplacian, $\mlap_D$, and its $L_2(\Omega)$-orthonormal
basis $e_1(x), e_2(x),\dots$ of eigenfunctions associated to the
usual ordering of its eigenvalues
$0<\lambda_1\le\lambda_2\le\dots$, which via Weyl's law for the counting function, cf.\ \cite[Ch.~6.4]{CuHi53}, gives
\begin{equation} \label{Weyl-id}
  \lambda_j={\cal O}(j^{2/n})\quad\text{ for $j\to\infty$}.
\end{equation}
This basis gives rise to a sequence of final value data
$u_{T,j}(x)=e_j(x)$ lying on the unit sphere in $L_2(\Omega)$ as
$\|u_{T,j}\|= \|e_j\|=1$ for $j\in\N$.
But the corresponding solutions to $u'\mlap u=0$,
i.e.\  $u_j(t,x)=e^{(T-t)\lambda_j}e_j(x)$,
have initial states $u(0,x)$ with $L_2$-norms that because of \eqref{Weyl-id} 
grow \emph{rapidly} with the index $j$,
\begin{equation}
  \|u_j(0,\cdot)\| = e^{T\lambda_j}\|e_j\| = e^{T\lambda_j}\nearrow\infty.
\end{equation}
This $L_2$-instability cannot be removed, of course, but it
only indicates that the $L_2(\Omega)$-norm is an insensitive choice for problem \eqref{heat-intro}. The
task is hence to obtain a norm on $u_T$ giving better control over the backward
calculations of $u(t,x)$---for the inhomogeneous heat problem \eqref{heat-intro}, 
an account of this was given in \cite{ChJo18ax}.
\end{remark}

\begin{remark}
Almog, Grebenkov, Helffer, Henry \cite{AlHe15,GrHelHen17,GrHe17} studied
the complex Airy operator $\mlap+\im x_1$ recently via triples $(H,V,a)$, leading to Dirichlet, Neumann, Robin and
transmission boundary conditions, in bounded and unbounded regions. 
Theorem~\ref{intro-thm} is expected to apply to final value problems for those of their
realisations that satisfy the coercivity condition in \eqref{coerciv-id}. However, 
$\mlap+\im x_1$ has empty spectrum on $\Rn$, cf.\ the fundamental paper of Herbst \cite{Her79}, so it
remains to be seen for which of the regions in \cite{AlHe15,GrHelHen17,GrHe17} there is a strictly descending
chain of domains as in \eqref{dom-intro}.
\end{remark}

\section{Preliminaries: Injectivity of Analytic Semigroups}
As indicated in the introduction, it is central to the analysis of final value problems that an
analytic semigroup of operators, like $e^{t\lap_D}$, always consists of \emph{injections}. This
shows up both at the technical and  conceptual level, that is, both in the proofs
and in the objects that enter the theorem.

A few aspects of semigroup theory in a complex Banach space $B$ is therefore recalled.
Besides classical references by Davies~\cite{Dav80}, Pazy~\cite{Paz83}, Tanabe~\cite{Tan79} or Yosida~\cite{Yos80},
a more recent account is given in \cite{ABHN11}.

The generator is
$\gen x=\lim_{t\to0^+}\frac1t(e^{t\gen}x-x)$, where $x$ belongs to the domain  $D(\gen)$ when the
limit exists. $\gen$ is a densely defined, closed linear 
operator in $B$ that for some $\omega \geq 0$, $M \geq 1$ satisfies the resolvent estimates
$\|(\gen-\lambda)^{-n}\|_{\B(B)}\le M/(\lambda-\omega)^n$ for $\lambda>\omega$, $n\in\N$.

The corresponding $C_0$-semigroup of operators $e^{t\gen}\in\B(B)$ is of type $(M,\omega)$: 
it fulfils that $e^{t\gen}e^{s \gen}=e^{(s+t)\gen}$ for $s,t\ge0$, $e^{0\gen}=I$ and 
$\lim_{t\to0^+}e^{t \gen}x=x$ for $x\in B$; whilst
\begin{align}  
  \|e^{t\gen}\|_{\B(B)} \leq M e^{\omega t} \quad \text{ for } 0 \leq t < \infty.
\end{align}
Indeed, the Laplace transformation 
$(\lambda I-\gen)^{-1}=\int_0^\infty e^{-t\lambda}e^{t\gen}\,dt$ 
gives a bijection of the semigroups of type $(M,\omega)$ onto (the resolvents of) the stated class of generators.

To elucidate the role of \emph{injectivity}, 
recall that if $e^{t\gen}$ is analytic, $u'=\gen u$, $u(0)=u_0$ is uniquely solved by
$u(t)=e^{t\gen}u_0$ for \emph{every} $u_0\in B$. 
Here injectivity of $e^{t\gen}$ is  equivalent to the important geometric property that the trajectories of two solutions
$e^{t\gen}v$ and $e^{t\gen}w$ of $u'=\gen u$  have no confluence point in $B$ for $v\ne w$.

Nevertheless, the literature seems to have focused on examples of semigroups with non-invertibility of
$e^{t\gen}$, like \cite[Ex.~2.2.1]{Paz83}; these necessarily concern non-analytic cases.
The well-known result below gives a criterion for $\gen$ to generate a
$C_0$-semigroup $e^{z\gen}$ that is defined and analytic for $z$ in the open sector
\begin{equation}
  S_{\theta}= \Set{z\in\C}{z\ne0,\ |\arg z | < \theta}.
\end{equation}
It is formulated  in terms of the spectral sector
\begin{equation} 
  \Sigma_{\theta} 
  =\left\{ 0 \right\}\cup \Set{ \lambda \in\C}{ |\arg\lambda | < \frac{\pi}{2} + \theta} .
\end{equation}

\begin{proposition}  \label{Pazy'-prop} 
If $\gen$ generates a $C_0$-semigroup of type $(M,\omega)$ and $\omega\in\rho(\gen)$,
the following properties are equivalent for each
$\theta \in\,]0,\frac{\pi}{2}[\,$:
\begin{itemize}
  \item[{\rm (i)}]
  The resolvent set $\rho(\gen)$ contains $\omega+\Sigma_{\theta}$  and 
\begin{equation} 
  \sup\Set{|\lambda-\omega|\cdot\|(\lambda I - \gen)^{-1} \|_{\B(B)}}{\lambda\in\omega+\Sigma_{\theta}, \
    \lambda \neq \omega} <\infty. 
\end{equation}
  \item[{\rm (ii)}] 
 The semigroup $e^{t \gen}$ extends to an analytic semigroup $e^{z \gen}$ defined for $z\in
 S_{\theta}$ with
\begin{equation}
   \sup\Set{ e^{-z\omega}\|e^{z\gen}\|_{\B(B)}}{z\in \overline{S}_{\theta'}}<\infty \quad \text{whenever $0<\theta'<\theta$}. 
\end{equation}
\end{itemize}
In the affirmative case, 
$e^{t \gen}$ is differentiable in $\B(B)$ for $t>0$ with derivative $(e^{t\gen})' = \gen
e^{t\gen}$, and for every $\eta$ such  that $\alpha(\gen)<\eta<\omega$ one has
\begin{align} \label{eta-est}
  \sup_{t>0} e^{-t\eta}\|e^{t\gen}\|_{\B(B)}+\sup_{t>0} te^{-t\eta}\|\gen e^{t\gen}\|_{\B(B)} <\infty,
\end{align}
 whereby $\alpha(\gen)=\sup\Re\sigma(\gen)$
denotes the spectral abscissa of $\gen$ (here $\alpha(\gen)<\omega$, as $0\in\Sigma_\theta$). 
\end{proposition}

In case $\omega=0$,
the equivalence is just a review of the main parts of Theorem~2.5.2 in \cite{Paz83}.
For general $\omega\ge0$, one can reduce to this case, since $\gen=\omega I+(\gen-\omega I)$ yields the operator
identity $e^{t\gen}=e^{t\omega}e^{t(\gen-\omega I)}$, where $e^{t(\gen-\omega I)}$ is of type
$(M,0)$ for some $M$. Indeed, the right-hand side is easily seen to be a $C_0$-semigroup, which
since $e^{t\omega}=1+t\omega+o(t)$ has 
$\gen$ as its generator, so the identity results from the bijectiveness of the Laplace transform.
In this way, (i)$\iff$(ii) follows straightforwardly from the case $\omega=0$, using for both implications that 
$e^{z\gen}=e^{z\omega}e^{z(\gen-\omega I)}$ holds in $S_\theta$ by unique analytic extension.

Since $\omega\in\rho(\gen)$ implies $\alpha(\gen)<\omega$, the above translation method gives 
$e^{t\gen}=e^{t\eta}e^{t(\gen-\eta I)}$, where $e^{t(\gen-\eta I)}$ is of type $(M,0)$ whenever
$\alpha(\gen)<\eta<\omega$. This yields the first part of \eqref{eta-est}, and the second now
follows from this and the case $\omega=0$ by means of the splitting $\gen=\eta' I+(\gen-\eta' I)$
for $\alpha(\gen)<\eta'<\eta$.

The reason for stating Proposition~\ref{Pazy'-prop} for general type $(M,\omega)$
semigroups is that it shows explicitly that cases with $\omega>0$ only have other estimates on
$\R_+$ or in
the closed subsectors $\overline{S}_{\theta'}$---but the mere analyticity in $S_{\theta}$ 
is unaffected by the translation by $\omega I$. Hence one has the following improved version of
\cite[Prop.\ 1]{ChJo18ax}:

\begin{proposition}  \label{inj-prop}
If a $C_0$-semigroup $e^{t\gen}$ of type $(M,\omega)$ on a complex Banach space $B$ has an
analytic extension $e^{z\gen}$ to
$S_{\theta}$ for some $\theta>0$, then $e^{z\gen}$ is \emph{injective} for every $z \in S_\theta$.
\end{proposition}

\begin{proof}
Let $e^{z_0 \gen} u_0 = 0$ hold for some $u_0 \in B$ and $z_0 \in S_{\theta}$.
Analyticity of $e^{z\gen}$ in $S_{\theta}$ carries over by the differential calculus in Banach
spaces to the map $f(z)= e^{z\gen}u_0$.
So for $z$ in a suitable open ball $B(z_0,r)\subset S_{\theta}$, a Taylor expansion and the identity
$f^{(n)}(z_0) = \gen^n e^{z_0 \gen}u_0$ for analytic semigroups (cf.~\cite[Lem.~2.4.2]{Paz83})  give
\begin{align}
  f(z) = \sum_{n=0}^{\infty} \frac{1}{n!}(z-z_0)^n f^{(n)}(z_0)=\sum_{n=0}^{\infty}
  \frac{1}{n!}(z-z_0)^n 
    \gen^n e^{z_0 \gen}u_0\equiv 0.
\end{align}
Hence $f\equiv 0$ on $S_{\theta}$ by unique analytic extension.
Now, as $e^{t\gen}$ is strongly continuous at $t=0$, we have
$u_0 = \lim_{t \rightarrow 0^+} e^{t\gen} u_0 = \lim_{t \rightarrow 0^+}f(t) = 0$.
Thus the null space of $e^{z_0\gen}$ is trivial.
\end{proof}

\begin{remark} 
The injectivity in Proposition~\ref{inj-prop} 
was claimed in \cite{Sho74} for  
$z>0$, $\theta\le \pi/4$ and $B$ a Hilbert space
(but not quite obtained, as noted in \cite[Rem.~1]{ChJo18ax}; cf.\ the details
in Lemma 3.1 and Remark 3 in \cite{18logconv}). A local version for the
Laplacian on $\Rn$ was given by Rauch \cite[Cor.~4.3.9]{Rau91}.
\end{remark}

As a consequence of the above injectivity, for an \emph{analytic} semigroup 
$e^{t\gen}$ we may consider its inverse that, consistently with the case in which $e^{t\gen}$ forms a group in
$\mathbb{B}(B)$, may be denoted for $t>0$ by
$e^{-t\gen} = (e^{t\gen})^{-1}$.
Clearly $e^{-t\gen}$ maps $D(e^{-t\gen})=R(e^{t\gen})$ bijectively onto $H$, and
it is  an unbounded, but closed operator in $B$. 

Specialising to a Hilbert space $B=H$, then also $(e^{t\gen})^*=e^{t\gen^*}$ is analytic, so 
$Z(e^{t\gen^*})=\{0\}$ holds for its null space by Proposition~\ref{inj-prop}; whence
$D(e^{-t\gen})$ is dense in $H$.
Some further basic properties are:

\begin{proposition}{\cite[Prop.\;2]{ChJo18ax}}  \label{inverse-prop}
The above inverses $e^{-t\gen}$ form a semigroup of unbounded operators in $H$,
\begin{equation} 
  e^{-s\gen}e^{-t\gen}= e^{-(s+t)\gen} \qquad \text{for $t, s\ge0$}.
\end{equation}
This extends to $(s,t)\in\R\times \,]-\infty,0]$, whereby $e^{-(t+s)\gen}$  may be unbounded for $t+s>0$. 
Moreover, as unbounded operators the $e^{-t\gen}$ commute with $e^{s \gen}\in \B(H)$, that is,
$e^{s \gen}e^{-t\gen}\subset e^{-t\gen}e^{s\gen}$ for $t,s\ge0$,
and have a descending chain of domains,
$H\supset  D(e^{-t\gen}) \supset D(e^{-t'\gen})$ for $0<t<t'$.
\end{proposition}

\begin{remark} \label{domain-rem}
The above domains serve as basic structures for the final value problem \eqref{heat-intro}.
They apply for $\gen=-A$ that generates an analytic semigroup $e^{-zA}$ in
$\B(H)$ defined in $S_{\theta}$ for $\theta=\operatorname{arccot}(C_3/C_4)>0$.  
Indeed, this was shown in \cite[Lem.\ 4]{ChJo18ax} with a concise argument using
$V$-ellipticity of $A$; the $V$-coercive case follows easily from
this via the formula $e^{-zA}=e^{kz}e^{-z(A+kI)}$ that results for $z\ge 0$ from the  translation trick
after Proposition~\ref{Pazy'-prop}; and then it defines $e^{-zA}$ by the right-hand side for every
$z\in S_\theta$.
(A rather more involved argument was given in \cite[Thm.\ 7.2.7]{Paz83} in a context of  uniformly
strongly elliptic differential operators.)
\end{remark}

\section{Proof of Theorem~\ref{intro-thm}}
  \label{aproof-sect}

To clarify a redundancy in the set-up, it is remarked here that in Proposition~\ref{LiMa-prop}
the solution space $X$ is a Banach space, which can have its norm in 
\eqref{eq:X} rewritten in  the following form, using the Sobolev space 
$H^1(0,T;V^*)=\Set{u\in L_2(0,T;V^*)}{\partial_t u\in L_2(0,T;V^*)}$,
\begin{align}  \label{eq:Xnorm}
  \|u\|_{X} = \big(\|u\|^2_{L_2(0,T;V)} + \sup_{0 \leq t \leq T}|u(t)|^2 + \|u\|^2_{H^1(0,T;V^*)}\big)^{1/2}.
\end{align}
Here there is a well-known inclusion $L_2(0,T;V)\cap H^1(0,T;V^*)\subset C([0,T];H)$ and an associated
Sobolev inequality for vector functions
(\cite{ChJo18ax} has an elementary proof)
  \begin{equation} \label{Sobolev-ineq}
  \sup_{0\le t\le T}| u(t)|^2\le (1+\frac{C_2^2}{C_1^2T})\int_0^T \|u(t)\|^2\,dt+\int_0^T \|u'(t)\|_*^2\,dt.
  \end{equation}
Hence one can safely omit the space $C([0,T];H)$ in  \eqref{eq:X} and remove
$\sup_{[0,T]}|u|$ from $\| \cdot\|_{X}$. Similarly $\int_0^T\|u(t)\|_*^2\,dt$ is redundant
in \eqref{eq:X} because $\|\cdot\|_*\le C_2\|\cdot\|$, so an equivalent norm on $X$ is given by
\begin{equation} \label{Xnorm-id}
  \vvvert u\vvvert_X =
  \big(\int_0^T \|u(t)\|^2\,dt +\int_0^T \|u'(t)\|_*^2\,dt\big)^{1/2}.   
\end{equation}
Thus $X$ is more precisely a Hilbertable space, as $V^*$ is so.
But the form given in \eqref{eq:X} is preferred in order to emphasize the properties of the solutions.

As a note on the equation $u'+Au=f$ with $u\in X$, the continuous function
$u\colon[0,T]\to H$ fulfils $u(t)\in V$ for a.e.\ $t\in\,]0,T[\,$, so the extension $A\in
\B(V,V^*)$ applies for a.e.\ $t$. Hence $Au(t)$ belongs to $L_2(0,T;V^*)$.

\subsection{Concerning Proposition~\ref{LiMa-prop}}

The existence and uniqueness statements in Proposition~\ref{LiMa-prop}  
are essentially special cases of the classical theory of Lions and Magenes, cf.\ \cite[Sect.~3.4.4]{LiMa72} on 
$t$-dependent $V$-elliptic forms $a(t;u,v)$. Indeed, because of the fixed final time
$T\in\,]0,\infty[\,$, their indicated extension to $V$-coercive forms works well here: since
$u\mapsto e^{\pm tk}u$ and $f\mapsto e^{\pm tk}$ are all bijections on $L_2(0,T;V)$ and $L_2(0,T;V^*)$,
respectively, the auxiliary problem $v'+(A+kI)v=e^{-kt}f$, $v(0)=u_0$ has a solution $v\in X$ according to the
statement for the $V$-elliptic operator $A+kI$ in \cite[Sect.~3.4.4]{LiMa72}, when $k$ is
the coercivity constant in \eqref{coerciv-id};
and since
multiplication by the scalar $e^{kt}$ commutes with $A$ for each $t$, it follows from the Leibniz rule
in $\cal{D}'(0,T;V^*)$ that the function $u(t)=e^{kt}v(t)$
is in $X$ and satisfies 
\begin{equation}
  u'+Au=f,\qquad u(0)=u_0.  
\end{equation}
Moreover, the uniqueness of a solution $u\in X$ follows from that of $v$, for if $u'+Au=0$,
$u(0)=0$, then it is seen at once that $v=e^{-kt}u$ solves $v'+(A+kI)v=0$, $v(0)=0$; so that
$v\equiv0$, hence $u\equiv0$.

In the $V$-elliptic case, the  well-posedness in Proposition~\ref{LiMa-prop} is a known corollary to the
proofs in \cite{LiMa72}. For coercive $A$, the above exponential
factors should also be handled, which can be done explicitly using

\begin{lemma}[Gr{\"o}nwall] \label{Gronwall-lem}
  When $\varphi$, $k$ and $E$ are positive Borel functions on $[0,T]$, and $E(t)$ is
  increasing, then validity on $[0,T]$ of the first of the following inequalities implies that of the
  second:
  \begin{equation}
    \varphi(t)\le E(t)+\int_0^t k(s)\varphi(s)\,ds\le E(t)\cdot\exp(\int_0^t k(s)\,ds).
  \end{equation}
\end{lemma}
The reader is referred to the proof of Lemma 6.3.6 in \cite{H97}, which actually covers the
slightly sharper statement above.
Using this, one finds in a classical way a detailed estimate on each subinterval $[0,t]$:

\begin{proposition}  \label{pest-prop}
The unique solution $u\in X$ of \eqref{ivA-intro}, cf.\
Proposition~\ref{LiMa-prop}, fulfils in terms of the boundedness and coercivity constants $C_3$, $C_4$
and $k$ of $a(\cdot,\cdot)$ that for $0\leq t\leq T$,
\begin{equation} \label{u-est}
  \begin{split}
  \int_0^t \|u(s)\|^2 \,ds  +\sup_{0\le s\le t}|u(s)|^2 &+  \int_0^t \|u'(s) \|_{*}^2 \,ds 
\\
  &\leq (2+ \frac{2C_3^2+C_4+1}{C_4^2}e^{2kt})\big(C_4|u_0|^2 + \int_0^t\|f(s)\|_{*}^2\,ds\big).
 \end{split}
\end{equation}
For $t=T$, this entails boundedness 
$ L_2(0,T;V^*)\oplus H\to X $
of the solution operator $(f,u_0)\mapsto u$.
\end{proposition}

\begin{proof}
As $u \in L_2(0,T;V)$, while $f$ and $A u$ and hence also $u'=f-Au$ belong to the dual space
$L_2(0,T;V^*)$, one has  in $L_1(0,T)$ the identity 
\begin{equation}
  \Re\dual{\partial_t u}{u} + \Re a(u,u) = \Re\dual{f}{u}.  
\end{equation}
Here a classical regularisation yields $\partial_t |u|^2=2\Re\dual{\partial_t u}{u}$, cf.\
\cite[Lem.\ III.1.2]{Tem84} or \cite[Lem.\ 2]{ChJo18ax}, 
so by Young's inequality and the $V$-coercivity,
\begin{align}
  \partial_t |u|^2 +  2(C_4 \|u\|^2 -k|u|^2)\leq 2 |\dual{f}{u}| \leq C_4^{-1} \|f\|_{*}^2 + C_4 \|u\|^2.
\end{align}
Integration of this yields,
since $|u|^2$ and $\partial_t |u|^2= 2\Re\dual{\partial_t u}{u}$ are in $L_1(0,T)$, 
\begin{align}  \label{VH-est}
  |u(t)|^2 + C_4 \int_0^t \|u(s)\|^2 \,ds  \leq |u_0|^2 + C_4^ {-1}\int_0^t\|f(s)\|_{*}^2\,ds +2k\int_0^t |u(s)|^2\,ds.
\end{align}
Ignoring the second term on the left-hand side, it follows from Lemma~\ref{Gronwall-lem} that, for
$0\le t\le T$, 
\begin{align} \label{H-est}
  |u(t)|^2 \leq \Big(|u_0|^2 + C_4^ {-1}\int_0^t\|f(s)\|_{*}^2\,ds\Big)\cdot\exp(2kt);
\end{align}
and since the right-hand side is increasing, one even has 
\begin{align} \label{H-est'}
 \sup_{0\le s\le t} |u(s)|^2 \leq \Big(|u_0|^2 + C_4^ {-1}\int_0^t\|f(s)\|_{*}^2\,ds\Big)\cdot\exp(2kt).
\end{align}

In addition it follows in a crude way, from \eqref{VH-est} and an integrated version of \eqref{H-est},  that 
\begin{equation} \label{V-est}
  \begin{split}
  C_4\int_0^t \|u(s)\|^2 \,ds  &\leq 
  \big(|u_0|^2 + C_4^ {-1}\int_0^t\|f(s)\|_{*}^2\,ds\big) 
  \big(1+\int_0^t (e^{2ks})'\,ds)
\\
  &= e^{2kt}\big(|u_0|^2 + C_4^ {-1}\int_0^t\|f(s)\|_{*}^2\,ds\big).
 \end{split}
\end{equation}

Moreover, as $u$ solves \eqref{ivA-intro}, it is clear that
$\|\partial_t u \|_{*}^2 \leq  (\|f\|_{*} + \|A u \|_{*})^2\leq  2\|f\|_{*}^2 + 2\|A u \|_{*}^2$,
and since $\|A\|\leq C_3$ holds for the norm in $\B(V,V^*)$, the above estimates entail
\begin{equation}  \label{V*-est}
  \begin{split}
  \int_0^t \|u'(s) \|_{*}^2 \,ds 
&\leq 2 \int_0^t \|f(s)\|_{*}^2 \,ds + 2 C_3^2  \int_0^t \|u(s)\|^2 \,ds 
\\
  &\leq 2(C_4+ \frac{C_3^2}{C_4} e^{2kt})\big(|u_0|^2 + C_4^ {-1}\int_0^t\|f(s)\|_{*}^2\,ds\big).
\end{split}  
\end{equation}
Finally the stated estimate \eqref{u-est} follows from \eqref{H-est'}, \eqref{V-est} and \eqref{V*-est}.
\end{proof}

\subsection{On the proof of the Duhamel formula}

As a preparation, a small technical result is recalled from Proposition 3 in \cite{ChJo18ax}, where a
detailed proof can be found:

\begin{lemma}  \label{Leibniz-lem}
  When $\gen$ generates an analytic semigroup on the complex Banach space $B$ and $w\in
  H^1(0,T;B)$, then the Leibniz rule
  \begin{equation}
    \partial_t e^{(T-t)\gen}w(t)= (-\gen)e^{(T-t)\gen}w(t)+e^{(T-t)\gen}\partial_t w(t)
  \end{equation}
  is valid in $\cal{D}'(0,T;B)$.
\end{lemma}

In Proposition~\ref{Duhamel-prop}, equation \eqref{u-id}
is of course just the Duhamel formula from analytic semigroup theory. However, since
$X$ also contains non-classical solutions, \eqref{u-id} requires a proof in the present
context---but as noted, it suffices just to reinforce the
classical argument by the injectivity of $e^{-tA}$ in Proposition~\ref{inj-prop}:

\begin{proof}[Proof of Proposition~\ref{Duhamel-prop}]
To address the last statement first, once \eqref{u-id} has been shown, Proposition~\ref{LiMa-prop} yields 
$e^{-tA}u_0\in X$ for $f=0$. For general $(f,u_0)$ one has  $u\in X$, so the last term
containing $f$ also belongs to $X$.

To obtain \eqref{u-id} in the above set-up, note that all terms in 
$\partial_t u+A u=f$ are in $L_2(0,T;V^*)$. Therefore  $e^{-(T-t)A}$
applies for a.e.\ $t\in[0,T]$ to both sides as an integration factor, so as an identity in
$L_2(0,T;V^*)$,
\begin{equation}
  \partial_t(e^{-(T-t)A}u(t))=e^{-(T-t)A}\partial_t u(t)+ e^{-(T-t)A}A u(t)=e^{-(T-t)A}f(t).
\end{equation}
Indeed, on the left-hand side $e^{-(T-t)A}u(t)$ is in $L_1(0,T;V^*)$ and 
its derivative in ${\cal D}'(0,T;V^*)$ follows the Leibniz rule in Lemma~\ref{Leibniz-lem}, since $u\in
H^1(0,T; V^*)$ as a member of $X$.

As $C([0,T];H)\subset L_2(0,T;V^*)\subset L_1(0,T;V^*)$, it is seen in the above that 
$e^{-(T-t)A}u(t)$ and $e^{-(T-t)A}f(t)$ both belong to $L_1(0,T;V^*)$. 
So when the Fundamental Theorem for vector functions 
(cf.\ \cite[Lem.\ III.1.1]{Tem84}, or \cite[Lem.\ 1]{ChJo18ax}) is applied and followed by use of the semigroup property and a 
commutation of $e^{-(T-t)A}$ with the integral, using Bochner's identity, cf.\ Remark~\ref{Bochner-rem} below, one finds that 
\begin{equation} \label{eq:identityT}
  \begin{split}
    e^{-(T-t)A}u(t)&=e^{-TA}u_0+\int_0^t e^{-(T-s)A}f(s)\,ds
\\
  &=e^{-(T-t)A}e^{-tA}u_0+e^{-(T-t)A}\int_0^t e^{-(t-s)A}f(s)\,ds. 
  \end{split}
\end{equation}
Since $e^{-(T-t)A}$ is linear and injective, cf.~Proposition~\ref{inj-prop}, equation \eqref{u-id} now results at once. 
\end{proof}

\begin{remark} \label{Bochner-rem}
  It is recalled that for $f\in L_1(0,T; B)$, where $B$ is a Banach space, it is a basic property 
  that for every functional $\varphi$ in the dual space $B'$, one  has Bochner's identity:
  $\dual{\int_0^T f(t)\,dt}{\varphi}= \int_0^T\dual{ f(t)}{\varphi}\,dt$.
\end{remark}

\subsection{Concerning Theorem~\ref{intro-thm}}

As all terms in \eqref{u-id} are in $C([0,T];H)$, it is safe to evaluate at $t=T$, which in view of
\eqref{yf-eq} gives that
$u(T)=e^{-TA}u(0)+y_f$. This is the flow map 
\begin{equation} \label{flow-id}
  u(0)\mapsto u(T).  
\end{equation}
Owing to the injectivity of $e^{-TA}$ once again, and that Duhamel's formula implies
$u(T)-y_f=e^{-TA}u(0)$, which clearly belongs to $D(e^{TA})$, this flow is inverted  by
\begin{equation}  \label{u0uT-id}
  u(0)=e^{TA}(u(T)-y_f).  
\end{equation}
In other words, not only are the solutions in $X$ to $u'+Au=f$ parametrised by the initial
states $u(0)$ in $H$ (for fixed $f$) according to Proposition~\ref{LiMa-prop}, but also the final
states $u(T)$ are parametrised by the $u(0)$. Departing from this observation, one may give an intuitive

\begin{proof}[Proof of Theorem~\ref{intro-thm}]
If \eqref{fvA-intro} is solved by $u \in X$, then $u(T)=u_T$ is reached from
the unique initial state $u(0)$ in \eqref{u0uT-id}. But the argument  for \eqref{u0uT-id}
showed that $u_T-y_f = e^{-TA} u(0)\in D(e^{TA})$, 
so \eqref{eq:cc-intro} is necessary. 

Given data $(f,u_T)$ that fulfill \eqref{eq:cc-intro},
then $u_0 = e^{TA}(u_T - y_f)$ is a well-defined vector in $H$, so 
Proposition~\ref{LiMa-prop} yields a
function $u\in X$ solving $u' +Au = f$ and $u(0)=u_0$. 
By the flow \eqref{flow-id},
this $u(t)$ has final state $u(T)=e^{-TA}e^{TA}(u_T-y_f)+y_f=u_T$, hence satisfies both equations in \eqref{fvA-intro}.
Thus \eqref{eq:cc-intro} suffices for solvability.

In the affirmative case, \eqref{eq:fvp_solution} results for any solution $u\in X$ by inserting
formula \eqref{u0uT-id} for $u(0)$ into \eqref{u-id}. 
Uniqueness of $u$ in $X$ is seen from the  right-hand side of \eqref{eq:fvp_solution}, where all
terms depend only on the given $f$, $u_T$, $A$ and $T>0$. 
That each term in \eqref{eq:fvp_solution}
is a function belonging to $X$ was seen in Proposition~\ref{Duhamel-prop}.

Moreover, the solution can be estimated in $X$ by substituting the expression \eqref{u0uT-id} for $u_0$
into the inequality in Proposition~\ref{pest-prop} for $t=T$.  For the norm in \eqref{Xnorm-id} this gives 
\begin{equation}
  \begin{split}
 \vvvert u\vvvert_X^2 &\leq (2+ \frac{2C_3^2+C_4+1}{C_4^2}e^{2kT})\max(C_4,1)\big(|u_0|^2 + \int_0^T\|f(s)\|_{*}^2\,ds\big)
\\
  &\le c(|e^{TA}(u_T-y_f)|^2+\|f\|_{L_2(0,T;V^*)}^2).
 \end{split}
\end{equation}
Here one may add $|u_T|^2$ on the right-hand side to arrive at the expression for $\|(f,u_T)\|_Y$  in Theorem 1.
\end{proof}

\begin{remark}
  It is easy to see from the definitions and proofs that $\cal{P}u=(\partial_t u+ Au, u(T))$ is a
bounded operator $X\to Y$.
The statement in Theorem~\ref{intro-thm} means that the solution operator $\cal{R}(f,u_T)=u$
(is well defined and) satisfies $\cal{P}\cal{R}=I$, but by the uniqueness also $\cal{R}\cal{P}=I$
holds. Hence $\cal{R}$ is a linear homeomorphism $Y\to X$.
\end{remark}

\section{The heat problem with the Neumann condition}
  \label{Neu-sect}

In the sequel $\Omega$ stands for a $C^\infty$ smooth, open bounded set in $\Rn$,
$n\ge2$ as described in \cite[App.~C]{G09}. In particular $\Omega$  is
locally on one side of its boundary $\Gamma=\partial \Omega$.
For such $\Omega$, the problem is to characterise the $u(t,x)$ satisfying
\begin{equation}  \label{heatN_fvp}
\left.
\begin{aligned}
  \partial_tu(t,x) -\Delta u(t,x) &= f(t,x) &&\text{ in } \, ]0,T[ \times \Omega 
\\
  \gamma_1 u(t,x) &= 0 && \text{ on } \, ]0,T[\, \times \Gamma
\\
   r_T u(x) &= u_T(x) && \text{ at } \left\{ T \right\} \times \Omega
\end{aligned}
\right\}
\end{equation}
While $r_Tu(x)=u(T,x)$,
the Neumann trace on $\Gamma$ is written in the operator notation $\gamma_1u=
(\nu\cdot\nabla u)|_{\Gamma}$, whereby $\nu$ is the outward pointing normal vector at $x\in\Gamma$. 
Similarly $\gamma_1$ is used for traces on $\, ]0,T[\, \times \Gamma$.

Moreover, $H^m(\Omega)$ denotes the usual Sobolev space that is normed by $\|u\|_m =
\big(\sum_{|\alpha|\le m}\int_\Omega |\partial^\alpha u|^2\,dx\big)^{1/2}$, which up to equivalent
norms equals the space $H^{m}(\overline{\Omega})$ of
restrictions to $\Omega$ of $H^{m}(\Rn)$ endowed with the infimum norm.  

Correspondingly the dual
of e.g.\ $H^1(\overline{\Omega})$ has an identification with the closed subspace of $H^{-1}(\Rn)$ given by the support condition in  
\begin{equation}
  H^{-1}_0(\overline{\Omega})=\Set{ u\in H^{-1}(\Rn)}{\operatorname{supp} u\subset \overline{\Omega}}.
\end{equation}
For these matters the reader is referred to \cite[App.\ B.2]{H}.
Chapter 6 and (9.25) in \cite{G09} could also be references for this and basic
facts on boundary value problems; cf.\ also \cite{Eva10, Rau91}.

The main result in Theorem~\ref{intro-thm} applies to \eqref{heatN_fvp} for 
$V= H^1(\overline\Omega)$,  $H = L_2(\Omega)$ and $V^* \simeq H^{-1}_0(\overline{\Omega})$, for
which there are inclusions
$H^1(\overline\Omega)\subset L_2(\Omega)\subset H^{-1}_0(\overline{\Omega})$, when $g\in
L_2(\Omega)$ via $e_\Omega$ (extension by zero outside of $\Omega$) is identified with
$e_\Omega g$ belonging to $H^{-1}_0(\overline{\Omega})$.
The Dirichlet form
\begin{align}  \label{sform-id}
  s(u,v) = \sum_{j=1}^n \scal{\partial_j u}{\partial_j v}_{L_2(\Omega)}
         = \sum_{j=1}^n \int_\Omega {\partial_j u}\overline{\partial_j v}\, dx 
\end{align}
satisfies $|s(v,w)|\le \|v\|_1\|w\|_1$, and the coercivity
in \eqref{coerciv-id} holds for $C_4=1$, $k=1$ since $s(v,v)=\|v\|_1^2-\|v\|_0^2$.

The induced Lax--Milgram operator is the Neumann realisation $\mlap_N$, which is selfadjoint due
to the symmetry of $s$ and has
its domain given by $D(\lap_N)=\Set{u\in H^2(\Omega)}{\gamma_1 u=0}$. 
This is a classical but non-trivial result (cf.\ the remarks prior to Theorem 4.28 in \cite{G09}, or Section
11.3 ff.\ there; or \cite{Rau91}).
Thus the homogeneous boundary condition is imposed via the condition $u(t)\in D(\lap_N)$ for $0<t<T$.

By the coercivity, $-A = \lap_N$ generates an analytic semigroup of injections
$e^{z\lap_N}$ in $\B(L_2(\Omega))$, and the bounded extension 
$\tilde\lap\colon H^{1}(\overline\Omega) \rightarrow H^{-1}_0(\overline{\Omega})$ 
induces the analytic semigroup $e^{z\tilde\lap}$ on $H^{-1}_0(\overline{\Omega})$; both
are defined for $z\in S_{\pi/4}$.
As previously, $(e^{t\lap_N})^{-1} = e^{-t\lap_N}$.

The action of $\tilde \lap$ is (slightly surprisingly) given by $\tilde\lap u=\dv(e_\Omega\grad u)$ for each $u\in
H^{1}(\overline\Omega)$, for when $w\in H^{1}(\Rn)$ coincides with
$v$ in $\Omega$, then \eqref{sform-id} gives
\begin{equation}
  \begin{split}
    \Dual{-\tilde\lap u}{v}=s(u,v)
   & = \sum_{j=1}^n \int_{\Rn} e_\Omega(\partial_j u)\cdot\overline{\partial_j w}\, dx
\\
   & =\sum_{j=1}^n\Dual{-\partial_j(e_\Omega\partial_j u)}{w}_{H^{-1}(\Rn)\times
    H^{1}(\Rn)}
\\
  &=\Dual{\sum_{j=1}^n-\partial_j(e_\Omega\partial_j u)}{v}_{H^{-1}_0(\overline{\Omega})\times H^{1}(\overline\Omega)}.
  \end{split}
\end{equation}
To make a further identification one may recall the formula $\partial_j(u\chi_\Omega)=(\partial_j
u)\chi_\Omega-\nu_j(\gamma_0u)dS$, valid for $u\in C^{1}(\Rn)$ when $\chi_\Omega$
denotes the characteristic function of $\Omega$, and $\gamma_0$, $S$ the restriction to $\Gamma$ and the surface
measure at $\Gamma$, respectively; a proof is given in \cite[Thm.\ 3.1.9]{H}. 
Replacing $u$ by $\partial_j u$ for some $u\in C^{2}(\overline\Omega)$, and using that $\nu(x)$ is
a smooth vector field around $\Gamma$, we obtain that
$\partial_j(e_\Omega\partial_ju)=e_\Omega(\partial_j^{2} u)-(\gamma_0\nu_j\partial_j u)dS$.
This now extends to all
$u\in H^{2}(\overline\Omega)$ by density and continuity,  and by summation one finds that in $\cal{D}'(\Rn)$,
\begin{equation} \label{tlap-id}
  \tilde\lap u=\dv(e_\Omega\grad u)=e_\Omega(\lap u)-(\gamma_1 u)dS.
\end{equation}
Clearly the last term vanishes for $u\in D(\lap_N)$; whence $\dv(e_\Omega\grad u)$ identifies in $\Omega$
with the $L_2$-function $\lap u$ for such $u$. But for general $u$ in the form domain $H^{1}(\overline\Omega)$,
none of the terms on the right-hand side make sense.

The solution space for \eqref{heatN_fvp} amounts to
\begin{equation}
\begin{split}
  X_0 &= L_2(0,T;H^1(\Omega)) \bigcap C([0,T]; L_2(\Omega)) \bigcap H^1(0,T; H^{-1}_0(\overline{\Omega})),
   \label{X0-id}
\\
  \|u\|_{X_0}&= \big(\int_0^T\|u(t)\|^2_{H^{1}(\Omega)}\,dt
\\
  &\hphantom{= \big(\int_0^T\|u(t)\|^2}
                  +\sup_{t\in[0,T]}\int_\Omega |u(x,t)|^2\,dx+
                 \int_0^T\|\partial_t u(t)\|^2_{H^{-1}_0(\overline{\Omega})}\,dt  \Big)^{1/2}.
\end{split}
\end{equation}
The corresponding data space is here given in terms of $y_f=\int_0^T e^{(T-t)\lap}f(t)\,dt$, cf.\
\eqref{yf-eq}, as
\begin{equation}
\begin{split}
  Y_0&= \left\{ (f,u_T) \in L_2(0,T;H^{-1}_0(\overline{\Omega})) \oplus L_2(\Omega) \Bigm|
                  u_T - y_f \in D(e^{-T\lap_N}) \right\},
\label{Y0-id}
\\
 \| (f,u_T) \|_{Y_0}
  &= \Big(\int_0^T\|f(t)\|^2_{H^{-1}_0(\overline{\Omega})}\,dt
\\
  &\hphantom{= \Big(\int_0^T\|u(t)\|^2}+ \int_\Omega\big(|u_T(x)|^2+|e^{-T\lap_N}(u_T - y_f )(x)|^2\big)\,dx\Big)^{1/2}.
\end{split}
\end{equation}
With this framework, Theorem~\ref{intro-thm} at once gives the following new result
on a classical problem:

\begin{theorem}  \label{heatN-thm}
Let $A=\mlap_N$ be the Neumann realization of the Laplacian in $L_2(\Omega)$ and 
$-\tilde\lap=-\dv(e_\Omega\grad\cdot)$ its extension $H^{1}(\overline{\Omega})\to H^{-1}_0(\overline{\Omega})$.
When  $u_T \in L_2(\Omega)$ and $f \in L_2(0,T;H^{-1}_0(\overline{\Omega}))$, 
then there exists a solution $u\in X_0$ of 
\begin{equation}
  \partial_t u-\dv(e_\Omega\grad u)=f,\qquad  r_Tu=u_T
\end{equation}
if and only if the data $(f,u_T)$ are given in $Y_0$, i.e.\ if and only if
\begin{equation}  \label{heat-ccc}
  u_T - \int_0^T e^{(T-s)\tilde\lap}f(s) \,ds\quad \text{ belongs to }\quad D(e^{-T \lap_N})=R(e^{T\lap_N}). 
\end{equation}
In the affirmative case, $u$ is uniquely determined in $X_0$ and satisfies the estimate
$\|u\|_{X_0} \leq c \| (f,u_T) \|_{Y_0}$.
It is given by the formula, in which all terms belong to $X_0$,
\begin{equation} 
  u(t) = e^{t\lap_N}e^{-T\lap_N}\Big(u_T-\int_0^T e^{(T-t)\tilde\lap}f(t)\,dt\Big) + \int_0^t e^{(t-s)\tilde\lap}f(s) \,ds.
\end{equation}
Furthermore the difference  in \eqref{heat-ccc} equals 
$e^{T\lap_N}u(0)$ in $L_2(\Omega)$. 
\end{theorem}

Besides the fact that $\tilde\lap=\dv(e_\Omega\grad\cdot)$ appears in the differential equation
(instead of $\lap$), it is noteworthy that there is no information on the boundary condition.
However, there is at least one simple remedy for this, for it is well known in analytic
semigroup theory, cf.\ \cite[Thm.\ 4.2.3]{Paz83} and \cite[Cor.\ 4.3.3]{Paz83}, that if the source term $f(t)$ is valued in $H$
and satisfies a global condition of H{\"o}lder continuity, that is, for some $\sigma\in\,]0,1[\,$,
\begin{equation}
  \sup\Set{|f(t)-f(s)|\cdot|t-s|^{-\sigma}}{0\le s<t\le T}<\infty,
\end{equation}
then the integral in \eqref{u-id} takes values in $D(A)$ for $0<t< T$ and
$A\int_0^t e^{-(t-s)A}f(s)\,ds$ is continuous $\,]0,T[\,\to H$. 

When this is applied in the above framework, the additional H{\"o}lder continuity yields $u(t)\in
D(\lap_N)=\Set{u\in H^2(\Omega)}{\gamma_1 u=0}$ for $t>0$, so the homogeneous Neumann condition is
fulfilled and $\tilde\lap u$ identifies with $\lap u$, as noted after \eqref{tlap-id}.
Therefore one has the following novelty:

\begin{theorem}
  If $u_T\in L_2(\Omega)$ and $f\colon\,[0,T]\to L_2(\Omega)$ is H{\"o}lder continuous of some
  order $\sigma\in\,]0,1[\,$, and if $u_T-y_f$ fulfils the criterion \eqref{heat-ccc}, then the
  homogeneous Neumann heat conduction final value problem \eqref{heatN_fvp} has a uniquely
  determined solution $u$ in $X_0$, satisfying $u(t)\in \Set{u\in H^2(\Omega)}{\gamma_1 u=0}$ for
  $t>0$, and depending continuously on $(f,u_T)$ in $Y_0$. Hence the problem is well posed in the
  sense of Hadamard. 
\end{theorem}

It would be desirable, of course, to show the well-posedness in a strong form, with an isomorphism
between the data and solution spaces.
\section{Final remarks}

\begin{remark} \label{GS-rem}
Grubb and Solonnikov~\cite{GrSo90} systematically treated
a large class of \emph{initial}-boundary problems of parabolic pseudo-differential equations
and worked out compatibility conditions characterising
well-posedness in full scales of anisotropic $L_2$-Sobolev spaces
(such conditions have a long history in the differential operator case, going back at least to work
of Lions and Magenes \cite{LiMa72} and Ladyzenskaya, Solonnikov and Ural'ceva \cite{LaSoUr68}).
Their conditions are explicit
and local at the curved corner $\Gamma\times\{0\}$, except for 
half-integer values of the smoothness $s$ that were shown to require so-called coincidence, which 
is expressed in integrals over the Cartesian product of the two boundaries $\{0\}\times\Omega$ and
$\,]0,T[\,\times\,\Gamma$; hence coincidence is also a non-local condition.  
Whilst the conditions of Grubb and Solonnikov are decisive for the solution's regularity, 
condition \eqref{eq:cc-intro} in Theorem~\ref{intro-thm} is in comparison crucial for the
\emph{existence} question. 
\end{remark}

\begin{remark}
Injectivity of the linear map $u(0)\mapsto u(T)$ for the homogeneneous equation $u'+Au=0$,
or equivalently its backwards uniqueness, was proved much earlier for problems with $t$-dependent
sesquilinear forms $a(t;u,v)$ by Lions and Malgrange~\cite{LiMl60}. In addition to some
$C^1$-regularity properties in $t$, they assumed that (the principal part of) $a(t;u,v)$ is
symmetric and uniformly $V$-coercive in the sense that $a(t;v,v)+\lambda\|v\|^2\ge \alpha\|v\|^2$
for certain fixed  $\lambda\in\R$, $\alpha>0$ and all $v\in V$. In Problem~3.4 of \cite{LiMl60}, they
asked whether backward uniqueness can be shown without assuming symmetry (i.e., for
non-selfadjoint operators $A(t)$ in the principal case), or more precisely under the hypothesis 
$\Re a(t;v,v)+\lambda\|v\|^2\ge \alpha\|v\|^2$. 
\linebreak[4]
The present paper gives an affirmative
answer for the $t$-independent case of their problem.
\end{remark}

\end{document}